\documentclass[12pt]{article} 
\usepackage{amssymb, amsmath, amsthm, upgreek, graphicx, placeins, mathrsfs} 
\usepackage{fullpage} 

\newcommand{\N}{\mathbb{N}}

\title{Identifying Self-Conjugate Partitions} 
\author{by Rebecca Odom} 
\date{\today} 

\newtheorem{thm}{Theorem}[section]
\newtheorem{prop}[thm]{Proposition}
\newtheorem{lem}[thm]{Lemma}
\newtheorem{cor}[thm]{Corollary}
\theoremstyle{definition}
\newtheorem{defn}[thm]{Definition}
\newtheorem{nex}{Counterexample}[section]

\theoremstyle{remark}
\newtheorem{ex}{Example}[section]

\begin{document}
\maketitle{ }
\section*{Abstract}
A \textit{partition} of a positive integer $n$ is defined
as a non-increasing sequence $P = [y_0,y_1,...,y_m]$ of positive integers which sum to $n$, where
the $y_i$ are called the \textit{parts} of the partition. A Young diagram is a visual representation of a
partition using rows of boxes, where each row of boxes corresponds to a part. The conjugate partition is similar
to a transpose of a matrix; we switch the rows with columns, or the index of a part with
the part itself. Self-conjugate partitions are partitions that are equal to their conjugate;
previously, the only known way to verify whether a partition is self-conjugate was through the use of
a Young diagram. In this research, by proving preliminary lemmas and theorems about
easily identifiable shapes which are symmetric, we come to the main result: by simply
adding the multiplicities of parts appropriately, we can show whether or not a partition is
self-conjugate without the use of a Young diagram.

\section{Introduction}

The term ``partition" usually means to parse out a mathematical object, such as a set. This idea, however, of parsing an object can also be applied to positive integers. The most natural way to do this is through integer
addition. For example, one can partition the integer  $n=14$ as $14=7+7=13+1=12+2$, and so on. In order to focus solely on the
summands of an integer partition and not the order in which they appear, we require the
summands to appear in descending order. The following is the formal definition of an
integer partition.
\begin{defn}\label{Def}
	A \textit{partition} of a positive integer $n$ is a sequence
	\[P=[y_0, y_1, ..., y_m]\]
	of positive integers where $m$ is a nonnegative integer, $y_0+y_1+...+ y_m=n$, and $y_i\geq y_{i+1}$ for all $0\leq i < m$. Each $y_i$ is called a \textit{part} of $P$, and $n=|P|:=y_0+y_1+...+ y_m$ is called the \textit{size} of $P$.
\end{defn}

Note that most textbooks use parentheses
instead of brackets when describing a partition. In this paper, we use brackets to improve
the readability of our proofs.\\

Below are some examples of partitions.

\begin{ex}
	We can partition the integer $n=7$ in the following fifteen ways:
	\begin{enumerate}
		\item[ ] $[7]:7=7$
		\item[ ] $[6,1]:7=6+1$
		\item[ ] $[5,2]:7=5+2$
		\item[ ] $[5,1,1]:7=5+1+1$
		\item[ ] $[4,3]:7=4+3$
		\item[ ] $[4,2,1]:7=4+2+1$
		\item[ ] $[4,1,1,1]:7=4+1+1+1$
		\item[ ] $[3,3,1]:7=3+3+1$
		\item[ ] $[3,2,2]:7=3+2+2$
		\item[ ] $[3,2,1,1]:7=3+2+1+1$
		\item[ ] $[3,1,1,1,1]:7=3+1+1+1+1$
		\item[ ] $[2,2,2,1]:7=2+2+2+1$
		\item[ ] $[2,2,1,1,1]:7=2+2+1+1+1$
		\item[ ] $[2,1,1,1,1,1]:7=2+1+1+1+1+1$
		\item[ ] $[1,1,1,1,1,1,1]:7=1+1+1+1+1+1+1$
	\end{enumerate}
\end{ex}
There are two ways to visualize a partition that are quite common. The first is a $Young\ diagram$ and the second is a $Ferrers\ diagram$. A Young diagram maps a partition of $n$ to a series of rows of boxes, where each part of the partition becomes a row. A Ferrers diagram is similar, but instead of mapping to rows of boxes, the partition maps to rows of dots or circles. Let us see an example. The graphics provided below (and every other figure to follow in the paper) were created by using \textit{Geogebra} \cite{MH}.

\begin{ex} \label{Ex2}
	Suppose we have the partition $P=[4,3,2,2,1]$ of $n=12$. Below are the two visual representations of $P$: the left diagram is a Young diagram, and the right is a Ferrers diagram.
	
	\begin{figure}[h]
		\centering
		\includegraphics[width=30mm]{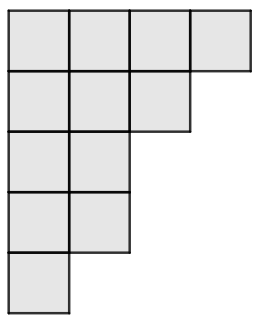}
		\includegraphics[width=4mm]{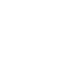}
		\includegraphics[width=29mm]{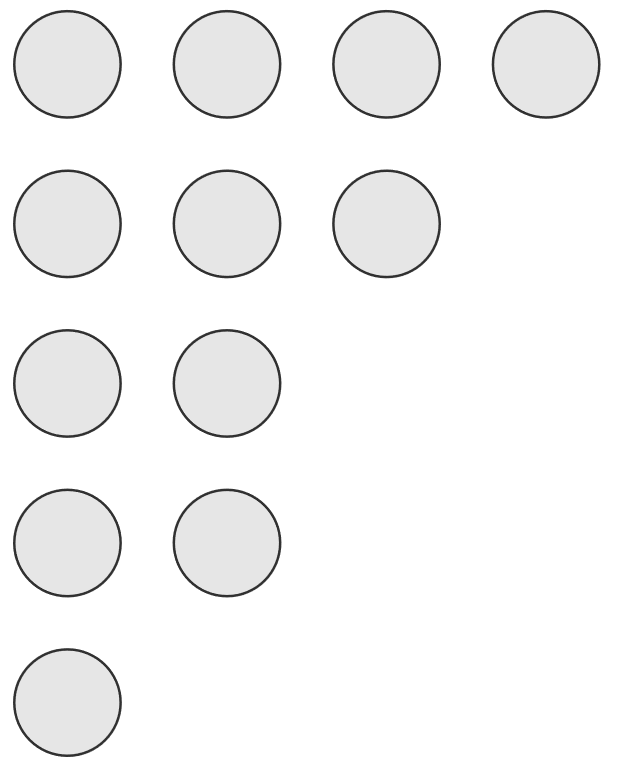}
	\end{figure}
	Notice how each row's number of boxes or dots corresponds to each part in the partition $P$ of $n$.
\end{ex}

\section*{}

There are quite a few things that we can do with partitions. To begin, let us define the partition function $p(n)$.

\begin{defn}
	The partition function $p(n)$ is the number of partitions of $n$.
\end{defn}

Now, $p(n)$ can be defined to be zero when $n<0$, and $p(0)=1$, which references the empty partition with no parts. The next six values of $p(n)$ are as follows, where we will temporarily use an exponent to show repeated values:
\begin{enumerate}
	\item[ ] $p(1)=1\ \to\ 1=[1].$
	\item[ ] $p(2)=2\ \to\ 2=[2],\ 1+1=[1^2].$
	\item[ ] $p(3)=3\ \to\ 3=[3],\ 2+1=[2,1],\ 1+1+1=[1^3]$
	\item[ ] $p(4)=5\ \to\ 4=[4],\ 3+1=[3,1],\ 2+2=[2^2],\ 2+1+1=[2,1^2],\ 1+1+1+1=[1^4].$
	\item[ ] $p(5)=7\ \to\ 5=[5],\ 4+1=[4,1],\ 3+2=[3,2],\ 3+1+1=[3,1^2],\ 2+2+1=[2^2,1],$
	\item[ ] $\ \ \ \ \ \ \ \ \ \ \ \ \ \ \ $ $\ 2+1+1+1=[2,1^3],\ 1+1+1+1+1=[1^5].$
	\item[ ] $p(6)=11\ \to\ 6=[6],\ 5+1=[5,1],\ 4+2=[4,2],\ 4+1+1=[4,1^2],\ 3+3=[3^2],$
	\item[ ] $\ \ \ \ \ \ \ \ \ \ \ \ \ \ \ \ \ \ $ $ 3+2+1=[3,2,1],\ 3+1+1+1=[3,1^3],\ 2+2+2=[2^3],$
	\item[ ] $\ \ \ \ \ \ \ \ \ \ \ \ \ \ \ \ \ \ \ \ \ \ \ $ $2+2+1+1=[2^2,1^2],\ 2+1+1+1+1=[2,1^4],\ 1+1+1+1+1+1=[1^6].$
\end{enumerate}

This function grows rapidly in value; for example \cite{Book}, ``$p(10)=42$, $p(20)=627$, $p(50)=204226$, $p(100)=190569292$, and $p(200)=3972999029388$." In fact, this leads us to the historical collaboration between G. H. Hardy and S. Ramanujan, who together showed what the value of $p(n)$ would be for any $n$ \cite{Paper}:

\begin{thm}
	The value of $p(n)$ is the integer nearest
	\[\frac{1}{2\sqrt{2}}\sum_{q=1}^{v}\sqrt{q}A_q(n)\psi_q(n),\]
	where $A_q(n)=\sum\omega_{p,q}e^{-2n\rho\pi i/q}$, the sum is over $p<q$ with $\gcd(p,q)=1$, $\omega_{p,q}$ is a certain 24$q$th root of unity, $v$ is of the order of $\sqrt{n}$, and
	\[\psi_q(n)=\frac{d}{dn}\left(\emph{exp}\left(\frac{C\sqrt{\left(n-\frac{1}{24}\right)}}{q}\right)\right),\ C=\pi\sqrt{\frac{2}{3}}\]
\end{thm}

This formula was later refined by H. Rademacher to be \cite{Third}:

\begin{thm}
	\[p(n)=\frac{1}{\pi\sqrt{2}}\sum_{k=1}^{\infty}A_k(n)k^{\frac{1}{2}}\left[\frac{d}{dx}\left(\frac{\sinh((\pi/k)(2/3\left(x-1/24\right))^\frac{1}{2})}{\left(x-1/24\right)^\frac{1}{2}}\right)\right]_{x=n},\]
	
	where
	
	\[A_k(n)=\sum_{h\ \emph{mod}\ k}^{\ }\omega_{h,k}e^{-2\pi i n h/k},\ \gcd(h,k)=1\ for\ h\ \emph{mod}\ k\]
\end{thm}

An example of using this second formula for $p(200)$ is as follows, where we compute the first eight terms:

\begin{enumerate}
	\item[ ] +3,972,998,993,185.896
	\item[ ] $\ \ \ \ \ \ \ \ \ \ \ \ \ \ $ +36,282.978
	\item[ ] $\ \ \ \ \ \ \ \ \ \ \ \ \ \ \ \ \ \ \ \ $ -87.555
	\item[ ] $\ \ \ \ \ \ \ \ \ \ \ \ \ \ \ \ \ \ \ \ $ +5.147
	\item[ ] $\ \ \ \ \ \ \ \ \ \ \ \ \ \ \ \ \ \ \ \ \ $+1.424
	\item[ ] $\ \ \ \ \ \ \ \ \ \ \ \ \ \ \ \ \ \ \ \ \ $+0.071
	\item[ ] $\ \ \ \ \ \ \ \ \ \ \ \ \ \ \ \ \ \ \ \ \ $+0.000
	\item[ ] $\ \ \ \ \ \ \ \ \ \ \ \ \ \ \ \ \ \ \ \ \ $+0.043
	\item[ ] \----\----\----\----\----\----\----\----\----\----
	\item[ ] $\ \ $ 3,972,99,029,388.004
\end{enumerate}

which is, within .004, the correct value for $p(200)$ \cite{Book}. However, the proofs for both formulas are quite lengthy, so we will omit them. Continuing on, the same Ramanujan and Hardy also discovered divisibility properties concerning $p(n)$. Ten of these divisibility properties are listed below \cite{Book}:

\begin{enumerate}
	\item[(1)] $p(4),\ p(9),\ p(14),\ p(19), ... \equiv 0\  ($mod$\ 5)$
	\item[(2)] $p(5),\ p(12),\ p(19),\ p(26), ... \equiv 0\  ($mod$\ 7)$
	\item[(3)] $p(6),\ p(17),\ p(28),\ p(39), ... \equiv 0\  ($mod$\ 11)$
	\item[(4)] $p(24),\ p(49),\ p(74),\ p(99), ... \equiv 0\  ($mod$\ 25)$
	\item[(5)] $p(19),\ p(54),\ p(89),\ p(124), ... \equiv 0\  ($mod$\ 35)$
	\item[(6)] $p(47),\ p(96),\ p(145),\ p(194), ... \equiv 0\  ($mod$\ 49)$
	\item[(7)] $p(39),\ p(94),\ p(149), ... \equiv 0\  ($mod$\ 55)$
	\item[(8)] $p(61),\ p(138), ... \equiv 0\  ($mod$\ 77)$
	\item[(9)] $p(116), ... \equiv 0\  ($mod$\ 121)$
	\item[(10)] $p(99), ... \equiv 0\  ($mod$\ 125)$
\end{enumerate}

This leads to Ramanujan's conjecture of the following theorem:
\begin{thm}
	If $\delta=5^a7^b11^c$ and $24\lambda\ \equiv \ 1\  (\emph{mod}\ \delta)$, then
	\[p(\lambda),\ p(\lambda+\delta),\ p(\lambda+2\delta),\ ...\ \equiv\ 0\ (\emph{mod}\ \delta).\]
\end{thm}

However, a counterexample to this conjecture was made by S. Chowla in the 1930s \cite{Book}. Specifically, he noticed that
\[p(243)=133978259344888\ \not\equiv\ 0\ (\textrm{mod}\ 7^3).\]
Though his conjecture was incorrect, Ramanujan did provide proofs (or sketches of proofs) for
\[p(5n+4)\ \equiv\ 0\ (\textrm{mod}\ 5),\]
\[p(7n+5)\ \equiv\ 0\ (\textrm{mod}\ 7),\]
\[p(25n+24)\ \equiv\ 0\ (\textrm{mod}\ 25),\]
\[p(49n+47)\ \equiv\ 0\ (\textrm{mod}\ 49).\]

Furthermore, in 1938, G. N. Watson showed that a ``modified version of the conjecture was true for all powers of 5 and 7" \cite{Book}. Later, in 1967 A. O. L. Atkin gave proof for the true, full modified conjecture:

\begin{thm}
	If $\delta=5^a7^b11^c$ and $24\lambda\ \equiv \ 1\  (\emph{mod}\ \delta)$, then $p(\lambda) \ \equiv \ 0 \ (\emph{mod}\ 5^a7^{[(b+2)/2]}11^c)$.
\end{thm}

Clearly, the value of $p(n)$ is more complicated than it first appears. Let us now consider what happens when we consider subsets of partitions \cite{Book}:

	Let $\varphi$ denote the set of all partitions, and let $p(S,n)$ denote the number of partitions of $n$ that belong to a subset $S\subseteq\varphi$. Furthermore, let $\mathcal{O}$ denote the set of all partitions with only odd parts and $\mathcal{D}$ denote the set of all partitions with distinct parts. Below we list partitions in $\mathcal{O}$ and $\mathcal{D}$ where we will, again, let an exponential represent the multiplicity of a part.

\begin{enumerate}
	\item[ ] $p(\mathcal{O},1)=1\ \to\ 1=[1].$
	\item[ ] $p(\mathcal{O},2)=1\ \to\ 1+1=[1^2].$
	\item[ ] $p(\mathcal{O},3)=2\ \to\ 3=[3],\ 1+1+1=[1^3].$
	\item[ ] $p(\mathcal{O},4)=2\ \to\ 3+1=[3,1],\ 1+1+1+1=[1^4].$
	\item[ ] $p(\mathcal{O},5)=3\ \to\ 5=[5],\ 3+1+1=[3,1^2],\ 1+1+1+1+1=[1^5].$
	\item[ ] $p(\mathcal{O},6)=4\ \to\ 5+1=[5,1],\ 3+3=[3^2],\ 3+1+1+1=[3,1^3],\ 1+1+1+1+1+1=[1^6].$
	\item[ ] $p(\mathcal{O},7)=5\ \to\ 7=[7],\ 5+1+1=[5,1^2],\ 3+3+1=[3^2,1],\ 1+1+1+1+1+1+1=[1^7].$
	\item[ ] $p(\mathcal{D},1)=1\ \to\ 1=[1].$
	\item[ ] $p(\mathcal{D},2)=1\ \to\ 2=[2].$
	\item[ ] $p(\mathcal{D},3)=2\ \to\ 3=[3],\ 2+1=[2,1].$
	\item[ ] $p(\mathcal{D},4)=2\ \to\ 4=[4],\ 3+1=[3,1].$
	\item[ ] $p(\mathcal{D},5)=3\ \to\ 5=[5],\ 4+1=[4,1],\ 3+2=[3,2].$
	\item[ ] $p(\mathcal{D},6)=4\ \to\ 6=[6],\ 5+1=[5,1],\ 4+2=[4,2],\ 3+2+1=[3,2,1].$
	\item[ ] $p(\mathcal{D},7)=5\ \to\ 7=[7],\ 6+1=[6,1],\ 5+2=[5,2],\ 4+3=[4,3],\ 4+2+1=[4,2,1].$
\end{enumerate}

Consider the fact that, for $n\leq7$, $p(\mathcal{O},n)=p(\mathcal{D},n)$. In fact, let us prove that this is true for all $n$. First, let us define generating functions for sequences and notation for subsets of positive integers:

\begin{defn}
	The \textit{generating function} $f(q)$ for the sequence $a_0$, $a_1$, $a_2$, $a_3$, ... is the power series $f(q)=\sum_{n\geq0}^{}a_nq^n$.
\end{defn}

\begin{defn}
	Let $\mathcal{H}$ be a set of positive integers. We let ``$\mathcal{H}$" denote the set of all partitions whose parts lie in $\mathcal{H}$. Consequently, $p(``\mathcal{H}",n)$ is the number of partitions of $n$ that have all their parts in $\mathcal{H}$. Thus, if $\mathcal{H}_0$ is the set of all odd positive integers, then ``$\mathcal{H}_0$"$=\mathcal{O}$. As in, $p($``$\mathcal{H}_0$"$,n)=p(\mathcal{O},n)$.
\end{defn}

\begin{defn}
	Let $\mathcal{H}$ be a set of positive integers. We let ``$\mathcal{H}$"$(\leq d)$ denote the set of all partitions in which no part appears more than $d$ times and each part is in $\mathcal{H}$. Thus, if $\mathbb{N}$ is the set of all positive integers, then  $p($``$\mathbb{N}$"$(\leq1),n)=p(\mathcal{D},n)$.
\end{defn}

This leads us to a theorem and its proof \cite{Book}.

\begin{thm}
	Let $\mathcal{H}$ be a set of positive integers, and let
	\[f(q)=\sum_{n\geq0}^{}p(``H",n)q^n,\]
	\[f_d(q)=\sum_{n\geq0}^{}p(``H"(\leq d),n)q^n.\]
	Then for $|q|<1$
	\begin{equation}
		f(q)=\prod_{n\in\mathcal{H}}^{}(1-q^n)^{-1}
	\end{equation}
	\[f_d(q)=\prod_{n\in\mathcal{H}}^{}(1+q^n+\dots+q^{dn})\]
	\begin{equation}
		\ \ \ \ \ \ \ \ \ \ =\prod_{n\in\mathcal{H}}^{}(1-q^{(d+1)n})(1-q^n)^{-1}
	\end{equation}
\end{thm}

Note that the two forms for $f_d(q)$ came from the sum of a finite geometric series:

\[1+x+x^2+\dots+x^r=\frac{1-x^{r+1}}{1-x}\ \  \cite{Book}.\]

\begin{proof}\cite{Book}
	We wish to prove (1) and (2). First, let us index all elements of $\mathcal{H}$ such that $\mathcal{H}={h_1, h_2, h_3, h_4,\dots}$; then we have
	\[\prod_{n\in\mathcal{H}}(1-q^n)^{-1}=\prod_{n\in\mathcal{H}}(1+q^n+q^{2n}+q^{3n}+\cdots)\]
	\[\ \ \ \ \ \ \ \ \ \ \ \ \ \ \ \ \ =(1+q^{h_1}+q^{2h_1}+q^{3h_1}+\cdots)\]
	\[\ \ \ \ \ \ \ \ \ \ \ \ \ \ \ \times\ (1+q^{h_2}+q^{2h_2}+q^{3h_2})\]
	\[\ \ \ \ \ \ \ \ \ \ \ \ \ \ \ \ \ \ \ \ \ \ \ \times\ (1+q^{h_3}+q^{2h_3}+q^{3h_3}+\cdots)\]
	\[\times\cdots\ \ \ \ \ \ \ \ \ \ \]
	\[\ \ \ \ \ \ \ \ \ \ \ \ \ \ \ \ \ \ \ \ \ \ \ =\sum_{a_1\geq0}\sum_{a_2\geq0}\sum_{a_3\geq0}\cdots q^{a_1h_1+a_2h_2+a_3h_3+\cdots}\]
	and the exponent of $q$ is the size of the partition $[\dots, h_3^{a_3}, h_2^{a_2}, h_1^{a_1}]$. Therefore, the $q^n$ term appears $p(``\mathcal{H}",n)$ times for each $n$ \cite{Book}. Thus,
	\[\prod_{n\in\mathcal{H}}(1-q^n)^{-1}=\sum_{n\geq0}p(``\mathcal{H}",n)q^n.\]
	Furthermore, the proof of (2) is identical, except we have a finite geometric series instead of an infinite one:
	\[\prod_{n\in\mathcal{H}}(1+q^n+q^{2n}+\dots+q^{dn})\]
	\[\ \ \ \ \ \ \ \ \ \ \ \ \ \ \ \ \ \ \ \ \ \ \ \ \ \ \ \ = \sum_{d\geq a_1\geq0}\ \sum_{d\geq a_2\geq0}\ \sum_{d\geq a_3\geq0}\cdots q^{a_1h_1+a_2h_2+a_3h_3+\cdots}\]
	\[\ \ \ =\sum_{n\geq0}p(``\mathcal{H}"(\leq d),n)q^n.\]
\end{proof}

\begin{cor}
	\emph{(Euler).} $p(\mathcal{O},n)=p(\mathcal{D},n)$ for all $n$.
\end{cor}

It is evident that the study of partitions is quite diverse, but the focus of this paper is the conjugate of a partition. If the reader is familiar with matrix methods or linear algebra, the concept will be familiar since it is similar to the transpose of a matrix. Taking the \textit{conjugate} of a partition is when every part of the partition is switched
with its index. This is similar to finding a transpose of a matrix, since we make every row of the Young diagram (or Ferrers diagram) a column and every column a row. This is equivalent to reflecting a partition diagram about its main diagonal through the first part's first box or dot. We see this visually in the next example.

\begin{ex}
	Suppose we are given the same partition $P=[4,3,2,2,1]$ as in Example \ref{Ex2}. We want to find its conjugate. We label each column below by color. The first column of $P$ is purple, and this will become the first part of the conjugate  $P'$ of some partition $P$ of $n$. The second column of $P$ is green, and this will become the second part of $P'$. We continue this process with every column, flipping the Young diagram of $P$ about its diagonal through the first box of the first part of $P$. Thus, we are left with $P'=[5,4,2,1]$.
	\begin{figure}[h]
		\centering
		\includegraphics[width=30mm]{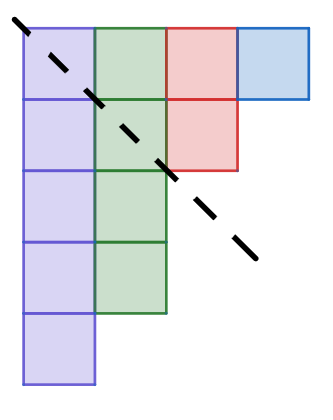}
		\includegraphics[width=20mm]{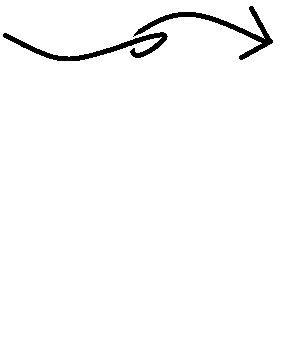}
		\includegraphics[width=38mm]{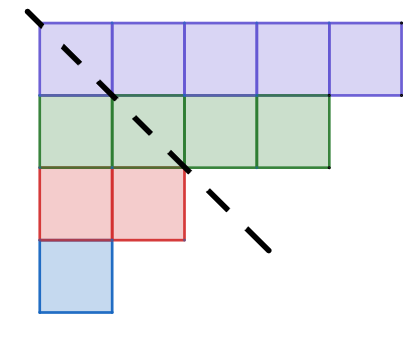}
		\caption{$P = [4,3,2,2,1]$ to the left; $P'$ to the right.}
		\label{figure1}
	\end{figure}
\end{ex}

Suppose we were to find the conjugate $P'$ of some partition $P$ of $n$. If $P=P'$, then we say $P$ is \textit{self-conjugate}. Let us see an example below:

\begin{ex}
	Suppose we have the partition $P=[4,3,2,1]$ of $n=10$. We find the conjugate of $P$ below, using the same method as before, and see that $P=P'$.
	\begin{figure}[h]
		\centering
		\includegraphics[width=30mm]{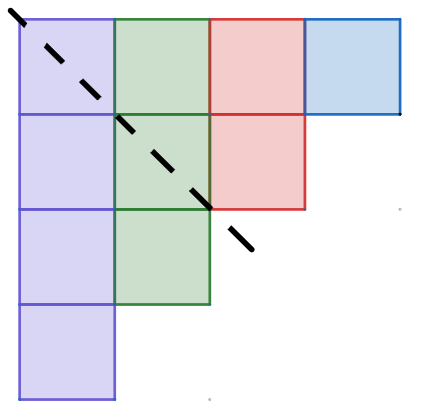}
		\includegraphics[width=20mm]{Redo5.PNG}
		\includegraphics[width=38mm]{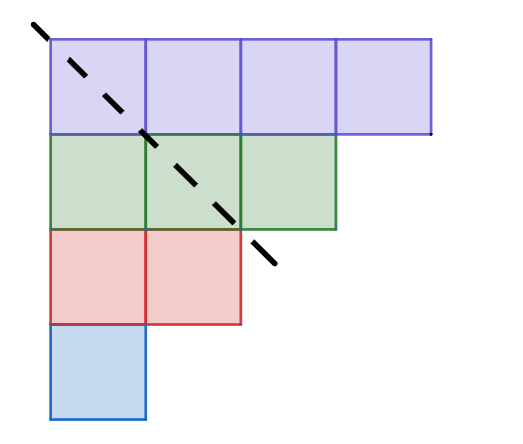}
		\caption{$P = [4,3,2,1]$ to the left; $P'$ to the right.}
		\label{figure2}
	\end{figure}
\end{ex}

This poses the question that this paper will answer: is it possible to determine if a partition $P$ is self-conjugate without the use of a Young or Ferrers diagram as a visual determinant? In particular, is there something about the parts themselves, a pattern to be found, that would let us know if $P$ is self-conjugate?

In Section \ref{Sec2}, we prove some preliminary
lemmas to begin our investigation of this question. In Section \ref{Sec3}, we prove an intermediate
theorem regarding symmetry of Young diagrams. In Section \ref{Sec4}, we introduce the idea
of partition addition and sub-partitions which will be essential in the proof of our main
theorem. In Section \ref{Sec5}, we discuss the idea of nesting and prove a few more lemmas about
symmetry of Young diagrams. Finally, in Section \ref{Sec6}, we state and prove our main theorem,
which gives a necessary and sufficient condition for a partition to be self-conjugate based
only on the parts and not on the Young diagram.

\section*{Acknowledgements}

Let it be known that without Dr. Madeline Dawsey from the University of Texas at Tyler, this paper and research would not have been possible. Dr. Dawsey is the one who started me on this journey through partitions and she has been my rock and foundation throughout all of this; I am forever grateful and indebted to her and the time she has given me. Furthermore, I would not be where I am today without several of the faculty of the mathematics department: Dr. Kassie Archer, Dr. Scott LaLonde, Dr. Katie Anders, Dr. Joseph Vandehey, Dr. Debbie Koslover, Dr. Sheldon Davis, Dr. Alex Bearden, Dr. Regan Beckham, Dr. David Milan, Dr. Christy Graves, and Dr. Steve Graves. Each of these professors are masters of their craft and have imbued me with the knowledge and skills to be successful both in mathematics and the world. I am far from the easiest student to instruct, and despite that, they all made my cup of success overflow. The mathematics department of the University of Texas at Tyler has my full gratitude and recommendation; there is no better place nor better group of fantastic people from which to learn the terrifying beauty that is the world of mathematics. Last, but not least, I would like to thank my grandparents and dedicate this paper to them: Benny and Linda Huff. In a full evaluation, it would be most obvious my life has been more full of turmoil than of joy, and yet, my grandparents have seen me through each dark moment; I owe every drop of success and goodness in me to the two of them.

\section{Simple Lemmas and Counting}\label{Sec2}

Before we introduce our lemmas and their proofs, let us see some examples and non-examples of self-conjugate partitions with their respective Young diagrams.

\begin{ex}\label{Ex3}
	$P=P'=[5,5,5,5,5]$ versus $Q=[6,5,5,5,5]\not=Q'=[5,5,5,5,5,1]$.
	\begin{figure}[h]
		\centering
		\includegraphics[width=30mm]{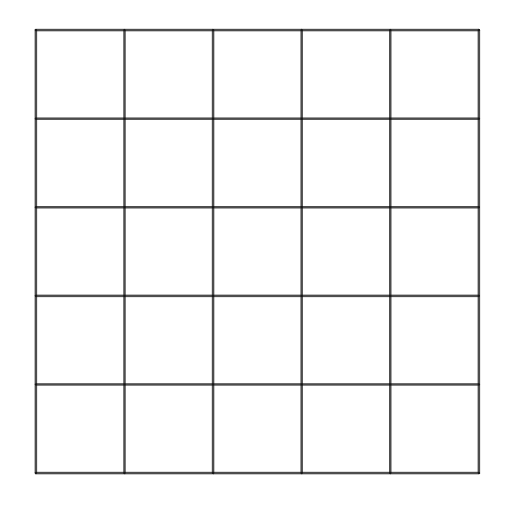}
		\includegraphics[width=31mm]{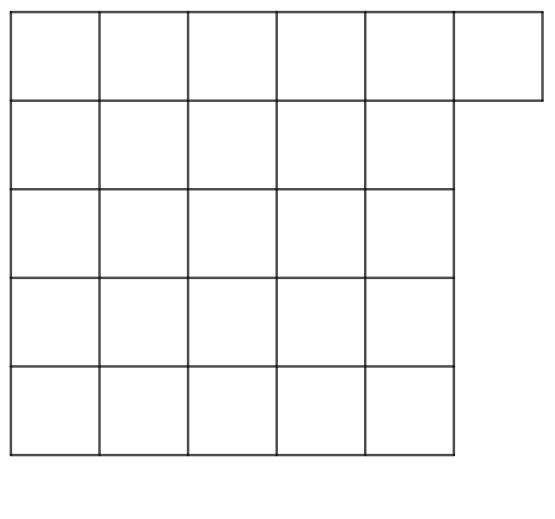}
		\includegraphics[width=25mm]{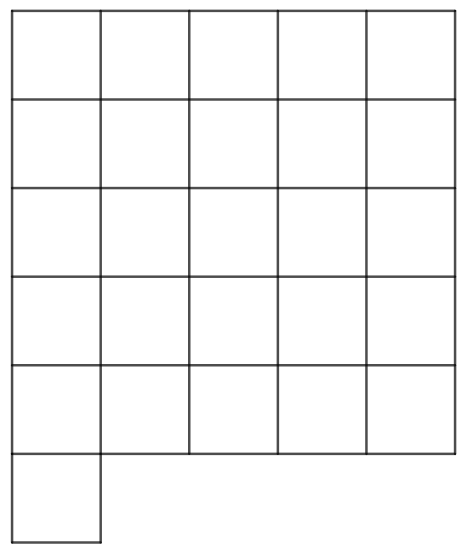}
		\caption{Left:$P = P'$, Center: $Q$, Right: $Q'$.}
		\label{figure3}
	\end{figure}
\end{ex}

\begin{ex}\label{Ex4}
	$P=P'=[5,4,3,2,1]$ versus $Q=[4,4,3,2,1]\not=Q'=[5,4,3,2]$.
	\begin{figure}[h]
		\centering
		\includegraphics[width=30mm]{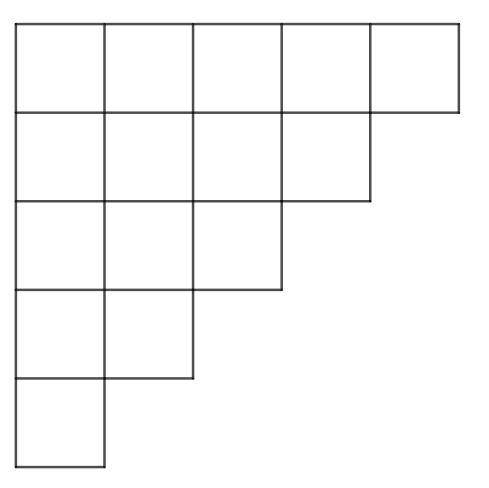}
		\includegraphics[width=30mm]{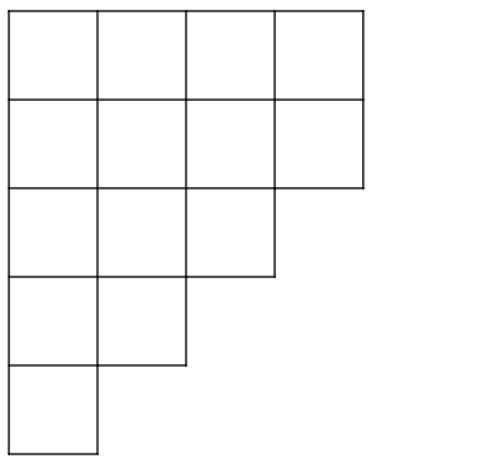}
		\includegraphics[width=30mm]{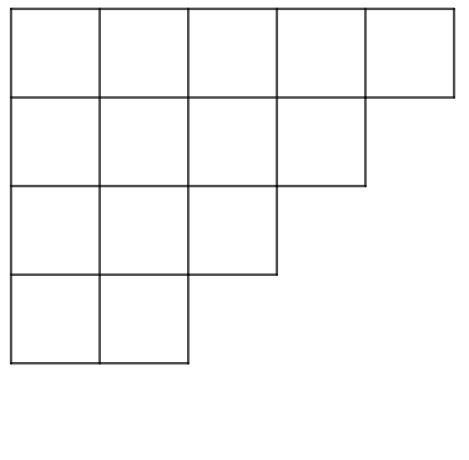}
		\caption{Left:$P = P'$, Center: $Q$, Right: $Q'$.}
		\label{figure4}
	\end{figure}
\end{ex}

\begin{ex}\label{Ex5}
	$P=P'=[5,1,1,1,1]$ versus $Q=[5,1,1,1]\not=Q'=[4,1,1,1,1]$.
	\begin{figure}[h]
		\centering
		\includegraphics[width=30mm]{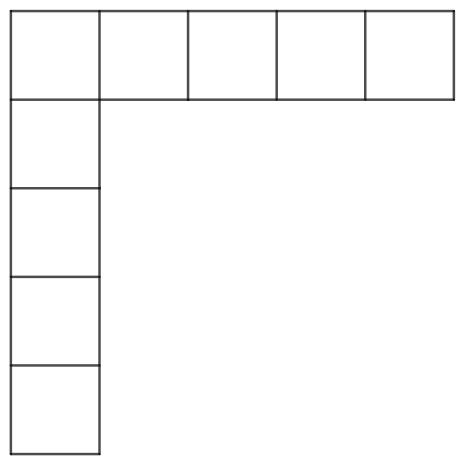}
		\includegraphics[width=30.5mm]{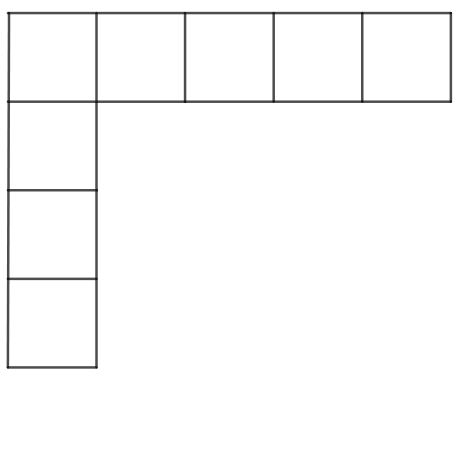}
		\includegraphics[width=31mm]{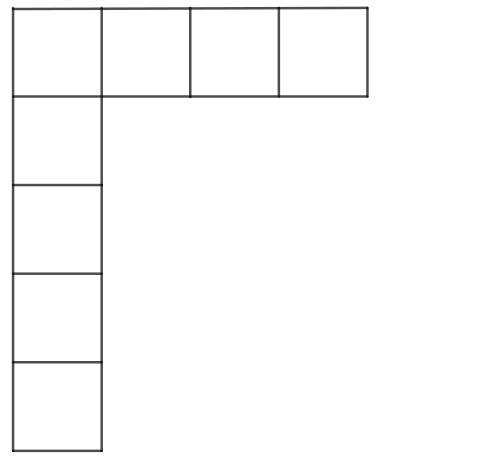}
		\caption{Left:$P = P'$, Center: $Q$, Right: $Q'$.}
		\label{figure5}
	\end{figure}
\end{ex}

In all of these examples, it is easy to see that the first part of $P$ must equal the number of parts of $P$ when $P=P'$; otherwise, the conjugate of our partition will not have the correct number of parts. This leads us to our first lemma.
\subsection{Size Equality}
\begin{lem} 
	\label{Equality}
	If a partition is self-conjugate, then
	the first part is equal to the number of parts.\\
\end{lem}

\begin{proof}
	Suppose $m,\ n\in\N$. Let $P=[n,\ a,\ b,...,\ c]$ be a partition that is self-conjugate. Let $P$ have \textit{length}
	$\ell(P) = m$ (as in, $P$ has $m$ parts).
	\begin{enumerate}
		\item[1.] Suppose $m>n$. Then the first row of the Young diagram will be $n$-long for the original partition $P$ and $m$-long for the conjugate partition $P'$. But, since $m>n$, this contradicts $P$ being self-conjugate since the first parts of the original and conjugate partition are not equal.
		\item[2.] Suppose $m<n$. We get a similar contradiction as before since the first row of the Young diagram will be $n$-long for the original partition $P$ and $m$-long for the conjugate partition $P'$, and thus, the first parts are not equal.
	\end{enumerate}
	Thus, since $m\not>n$ and $m\not<n$, and $m,\ n\in\N$, then we have $m=n$ when $P$ is a self-conjugate.\\
\end{proof}

From now on, we only want to consider partitions of dimension $d\in\N$, where we say a partition has $dimension$ $d$ if the first part equals the number of parts equals $d$. Lemma \ref{Equality} is not an if and only if statement, but we do not want to consider partitions where the number of parts does not equal the first part, because those partitions clearly cannot be self-conjugate. All partitions of dimension $d$ that are self-conjugate fit into a square, $d\times d$ Young diagram. In fact, the following section provides a formula for the number of dimension $d$ partitions for any positive integer $d$.

\subsection{Counting}
In the world of mathematics, we often concern ourselves with how many mathematical objects we have of certain types. We now wish to see how many self-conjugate partitions of given a dimension $d$ there are. First, we count (where the main color represents the partition and the secondary color represents what we have taken away from the square $d\times d$ Young diagram to obtain the partition):\\
\begin{enumerate}
	\item[$d=1:$] 
	\begin{figure}[h]
		\centering
		\includegraphics[width=5mm]{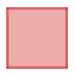}
		\caption{Self-conjugate partition of dimension 1.}
	\end{figure}
	There are $1=2^0=2^{1-1}$ self-conjugate partitions (red).\\
	
	\item[$d=2:$] 
	\begin{figure}[h]
		\centering
		\includegraphics[width=10mm]{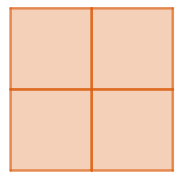}
		\includegraphics[width=10mm]{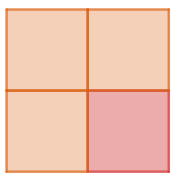}
		\caption{Self-conjugate partitions of dimension 2.}
	\end{figure}
	There are $2=2^1=2^{2-1}$ self-conjugate partitions (orange).\\
	
	\item[$d=3:$] 
	\begin{figure}[h]
		\centering
		\includegraphics[width=15mm]{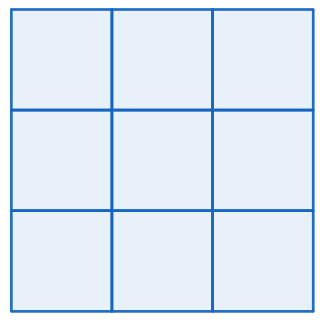}
		\includegraphics[width=15mm]{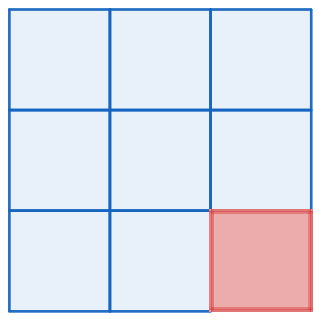}
		\includegraphics[width=15mm]{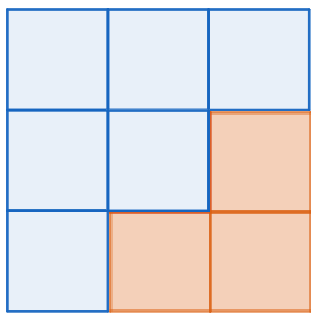}
		\includegraphics[width=15mm]{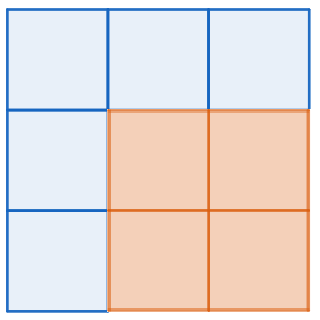}
		\caption{Self-conjugate partitions of dimension 3.}
	\end{figure}
	There are $4=2^2=2^{3-1}$ self-conjugate partitions (blue).\\
	
	\item[$d=4:$] 
	\begin{figure}[h]
		\centering
		\includegraphics[width=20mm]{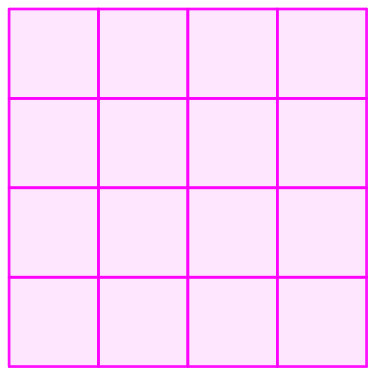}
		\includegraphics[width=20mm]{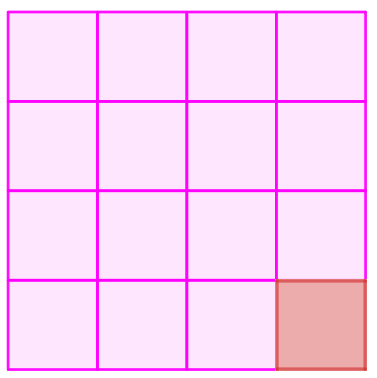}
		\includegraphics[width=20mm]{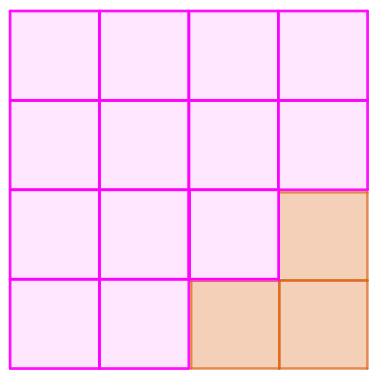}
		\includegraphics[width=20mm]{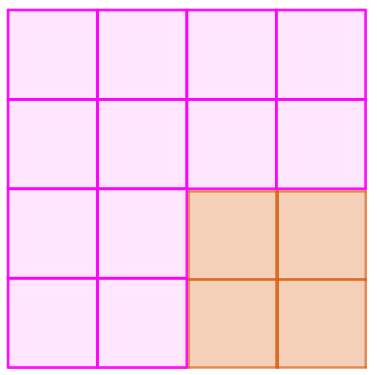}\\
		\includegraphics[width=20mm]{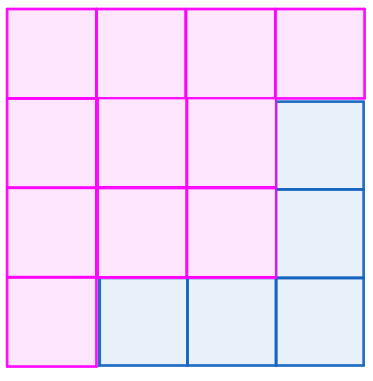}
		\includegraphics[width=20mm]{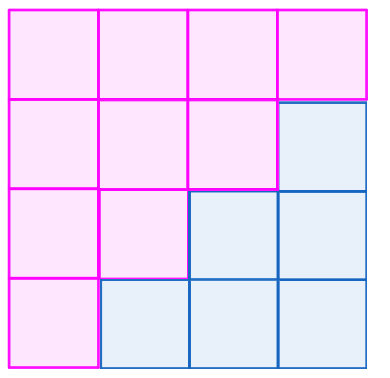}
		\includegraphics[width=20mm]{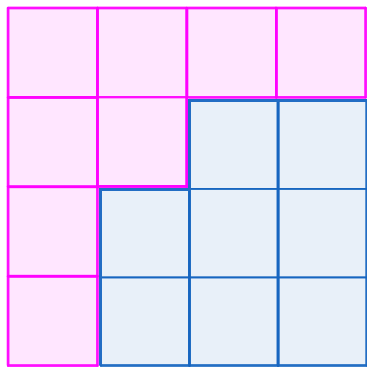}
		\includegraphics[width=20mm]{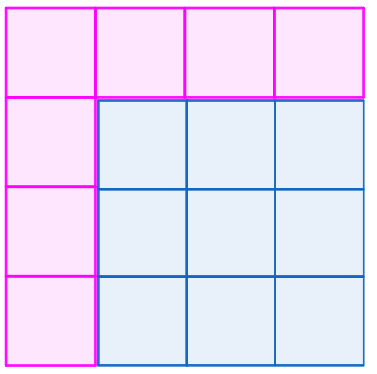}
		\caption{Self-conjugate partitions of dimension 4.}
	\end{figure}
	There are $8=2^3=2^{4-1}$ self-conjugate partitions (pink).\\
\end{enumerate}

The following theorem provides a formula for the number of sel-conjugate partitions of dimension $d\geq 3$.
\begin{thm}
	There are 
	\[2^{d-1}=2+\sum_{i=1}^{d-2}2^i\]
	self-conjugate partitions of dimension $d$, for $d\geq3$.
\end{thm}

\begin{proof}
	First, let us prove that $2^k=2+\sum_{i=1}^{k-1}2^i$. Notice:
	\begin{enumerate}
		\item[$\ $] $2+\sum_{i=1}^{k-1}2^i$
		\item[$=$] $2+2+\sum_{i=2}^{k-1}2^i$
		\item[$=$] $4+\sum_{i=2}^{k-1}2^i=2^2+4+\sum_{i=3}^{k-1}2^i$
		\item[$=$] $8+\sum_{i=3}^{k-1}2^i=2^3+\sum_{i=3}^{k-1}2^i$
		\item[$\ $] $\vdots$
		\item[$=$] $2^{k-2}+\sum_{i=k-2}^{k-1}2^i=2^{k-2}+2^{k-2}+\sum_{i=k-1}^{k-1}2^i$
		\item[$=$] $2(2^{k-2})+2^{k-1}=2^{k-1}+2^{k-1}=2(2^{k-1})$
		\item[$=$] $2^k$
	\end{enumerate}
	
	Thus, the claim that $2^k=2+\sum_{i=1}^{k-1}2^i$ holds true (for $k\geq2$). Now, we must show that for dimension $d$, there are $2^{d-1}$ self-conjugate partitions. From the pictures above, there are two things that are readily clear: for every dimension $d\geq2$, there will be a partition that has no squares removed and a partition that has only one square removed from its Young diagram. In other words, for dimension $d\geq2$,  we will have partitions of the integers $d^2$ and $d^2-1$, which each have unique representations in a given dimension $d$. The rest of the self-conjugate partitions relate to the rest of the previous dimensions, where we remove squares in the pattern of the previous self-conjugate partitions. Thus, let us say $x_d$ is the number of self-conjugate partitions of dimension $d$; we have $x_d=x_{d-1}+x_{d-2}+x_{d-3}+...+x_{d-(d-3)}+x_{d-(d-2)}+2$. So, our proof becomes a proof by induction.
	\begin{enumerate}
		\item[] \textit{Base case}: Suppose $d=3$. From the work before, we have that there are $2^2=2+\sum_{i=1}^{2-1}2^i=2+2=4$ self-conjugate partitions.
		\item[] \textit{Induction}: Suppose for all $d=j\leq k$, there are $2^{j-1}$ self-conjugate partitions of dimension $d=j$. We wish to find how many self-conjugate partitions there are for $d=k+1$. From our work before, we know $2^{k-1}=2+\sum_{i=1}^{k-2}2^i$. We also know that the number of self-conjugate partitions of dimension $d=k+1$ is $x_{k+1}=2+\sum_{j=2}^{k}x_j$. Thus, we write:
		\begin{enumerate}
			\item[$x_{k+1}$] $=2+\sum_{j=2}^{k}x_j$
			\item[$\ $] $=2+x_k+\sum_{j=2}^{k-1}x_j$
			\item[$\ $] $=2+2^{k-1}+\sum_{j=2}^{k-1}x_j$
			\item[$\ $] $=2+2^{k-1}+x_{k-1}+\sum_{j=2}^{k-2}x_j$
			\item[$\ $] $=2+2^{k-1}+2^{k-2}+\sum_{j=2}^{k-2}x_j$
			
			\item[$\ $] $\vdots$

			\item[$\ $] $=2+2^{k-1}+2^{k-2}+2^{k-3}+...+2^{k-(k-2)}+2^{k-(k-1)}$
			\item[$\ $] $=2+\sum_{i=1}^{k-1}2^i$
			\item[$\ $] $=2^k$
		\end{enumerate}
	\end{enumerate}
	This proves the theorem.
\end{proof}

One special type of partition which we noticed above and will discuss next is a partition whose Young diagram is a perfect square. This type of partition is related to the \textit{Durfee square} of a partition, the largest perfect square that can fit inside the Young diagram of the partition. In particular, we will focus next on partitions whose Young diagrams are equal to their own Durfee squares. We will refer to this type of partition as a ``pure Durfee square", and we will show that pure Durfee squares of any dimension are self-conjugate.

\subsection{Durfee Square}
\begin{lem} 
	\label{Durfee}
	Suppose the first part of a partition $P$ is equal to the number of parts of $P$, and all parts of $P$ are equal; namely, $P$ is a pure Durfee square. Then $P$ is self-conjugate.\\
\end{lem}

\begin{proof}
	Suppose $P$ is a pure Durfee square. Then we have that $P=[n,\ n,...,\ n]$ with $n$ parts. This means all parts of the conjugate partition $P'$ will be of size $n$, since every column and every row equals $n$. Thus, $P=P'=[n,\ n,...,\ n]$, so $P$ is a self-conjugate partition.\\
\end{proof}

The self-conjugate partition $P$ from Example 2.2 is also a special type of self-conjugate partition: a ``fancy triangle". In the pure Durfee square case, $P=[n,\ n,...,\ n]$ with $n$ parts is self-conjugate, whereas in the fancy triangle case, the pattern that holds is a series of perfectly descending parts. Explicitly, a partition $P$ is a fancy triangle if $P = [n,n-1,n-2,...,2,1]$. We now show this type of partition is always self-conjugate as well. 

\subsection{Fancy Triangle}
\begin{lem} 
	\label{Triangle}
	 Suppose the first part of a partition $P$ is equal to the number of parts, and each subsequent part is exactly one less than the previous part; namely, $P$ is a fancy triangle. Then P is self-conjugate.
\\
\end{lem}

\begin{proof}
	Suppose $P$ is a fancy triangle: 
	\[P=[n,\ n-1, \ n-2, ...,\ n-(n-1)].\] 
	The first row of the Young diagrams of both $P$ and its conjugate $P'$ are $n$-long by assumption. Thus, there is symmetry along the first (outer) row and column of the Young diagram.
	Now, remove the first row and column of $P$ to form the new partition
	\[Q=[(n-1),\ (n-1)-1,\ (n-1)-2,...,\ (n-1)-((n-1)-1)].\]
	Let $m=n-1$. Then 
	\[Q=[m,\ m-1,\ m-2,...,\ m-(m-1)].\]
	So, $Q$ has $m=n-1$ parts and each part is exactly one less than the previous. Thus, $Q$ is another fancy
	triangle, and the first row of
	the Young diagrams of both $Q$ and its conjugate $Q'$ are $m$-long, so $P$ has symmetry
	along the first (outer) two rows and columns.
	Continue inductively until we get to one of two cases depending on whether $n$ is even or odd: a partition $Z$ that either is $Z=[1]$ or $Z=[2,1]$. Both cases are trivially self-conjugate, since $Z=[1]=Z'$ is symmetric by Lemma \ref{Durfee} and $Z=[2,1]=Z'$ is the base case of a symmetric fancy triangle.
	Thus, the claim of Lemma \ref{Triangle} holds true.\\
\end{proof}

There is one last special self-conjugate partition we wish to discuss: the ``fancy L". This partition only has one row and one column in its Young diagram, and both have equal length. Thus, we will have a partition with the first part $n$ and all other parts 1. We saw an example of this in Example \ref{Ex5}.

\subsection{Fancy L}
\begin{lem} 
	\label{L}
	Suppose $P = [n, 1, 1, . . . , 1]$ with $n$ parts; namely, $P$ is a fancy
	L. Then $P$ is self-conjugate.\\
\end{lem}

\begin{proof}
	Suppose $P = [n, 1, 1, . . . , 1]$
	with $n$ parts. The first row of the Young diagram of the conjugate $P'$
	is $n$-long because
	$P$ has $n$ parts. Furthermore, $P'$ has $n$ parts, with one part of size $n$ followed by $n-1$ parts of size 1, because only the first row of the Young diagram of $P$ is $n$-long and the rest of the parts are of size 1. Thus, the first (outer) row and column of $P$ is symmetric. Since there are no other boxes to consider in the Young diagrams of either $P$ or $P'$, our claim for Lemma \ref{L} holds true.\\
\end{proof}

\section{Piecewise Theorem}\label{Sec3}

We now change direction slightly and prove a seemingly unrelated theorem. The idea of this section is that if a partition is self-conjugate, then the area on either side of the diagonal in a Young diagram is equal; this theorem and its proof will help us in the proof of the main theorem later.

\begin{thm} 
	\label{Piecewise}
	Let $\phi$ be a map on the set of all partitions such that if\\ $P=[z_0,z_1,...,z_m]$, then $\phi(P)$ is given by the piecewise function
	\[\phi:\ P\to f(x)=
	\begin{cases} 
		z_0 & \mathrm{if\ }0\leq x <1 \\
		z_1 & \mathrm{if\ }1\leq x <2 \\
		\vdots & \ \ \\
		z_m & \mathrm{if\ }m-1\leq x <m. 
	\end{cases}\]
	
	Let $g(x) = x$. If $P$ is self-conjugate, then the area of the region bounded by $f$, $g$, and the $x$-axis is equal to the area of the region bounded by $f$, $g$, and the $y$-axis.\\
\end{thm}

\begin{proof}
	Suppose $P$ is a self-conjugate partition, and suppose you calculated the areas of the regions defined above by integration.\\
	\begin{figure}[h]
		\centering
		\includegraphics[width=120mm]{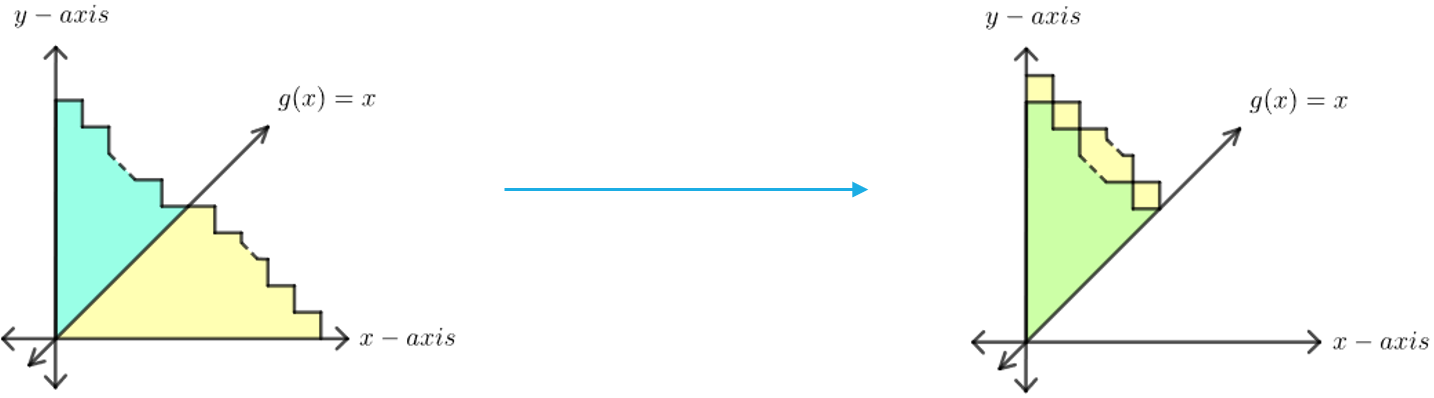}
		\caption{Superimposing areas.}
		\label{figure6}
	\end{figure}

	Assume by way of contradiction that
	when you superimpose one region on the other, by reflecting one across $g$, one has
	more area. Without loss of generality, assume the region bounded by the x-axis has
	more area.
	
	\begin{figure}[h]
		\centering
		\includegraphics[width=90mm]{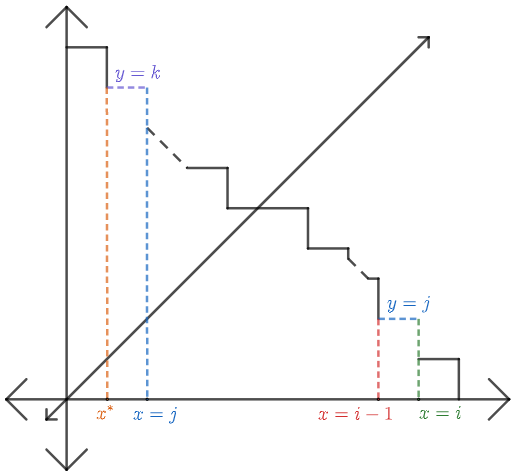}
		\caption{Comparing bounds.}
		\label{figure7}
	\end{figure}
	
	Either the first part does not equal the number of parts in $P$, failing Lemma \ref{Equality}; or, for some $i\in\{1,\ 2,...,\ m\}$, and for  $(i-1)\leq x <i$, we have $f(x)=j$, and  for some $(j-1)\leq x\textrm{*}<j$ and $j\leq(i-1)$, we have $f(x\textrm{*})=k$, where $k<i$ (see Figure \ref{figure7} for an illustration).
	
	This implies the conjugate of $P$, $P'$, will not equal $P$, since $f(x\textrm{*})=k$ for $P$, but $f(x\textrm{*})=i$ for $P'$ and $k<i$. Thus, we have have a contradiction of our assumption that $P=P'$.
	
	Similarly, if the region bounded by the $y$-axis had more area, we would draw a similar conclusion. Thus, if the areas are unequal, $P$ cannot be a self-conjugate.
\end{proof}

The following is another symmetry
property shared by all self-conjugate partitions.

\begin{prop}
	Let $\phi$ be the map defined in Theorem \ref{Piecewise}. If $P$ is self-conjugate, then $f(x)=f^{-1}(x)$.\\
\end{prop}

Note that the converse of Theorem \ref{Piecewise} is not true. Having equal area on either side of the diagonal does \textit{not} imply that the partition is self-conjugate. Consider the counterexample below.
\begin{nex}
	Suppose we have the partition $P=[3,3,1,1]$. We graph $\phi(P)$ below on the Cartesian plane.
	\begin{figure}[h]
		\centering
		\includegraphics[width=60mm]{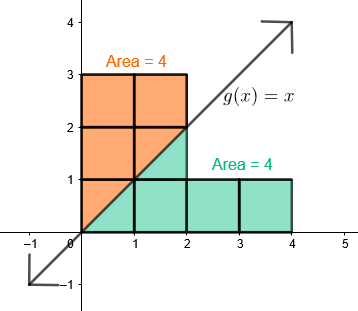}
	\end{figure}
	Notice that the areas on either side of the diagonal $g(x)=x$ (orange and teal) are equal, but the partition is not self-conjugate. Thus, equal areas on either side of the diagonal cannot imply a self-conjugate partition, so the relationship is not ``if and only if".
\end{nex}

\section{Addition and Sub-Partitions}\label{Sec4}
We now introduce two definitions with some examples. Being able to add partitions and to sub-partition as defined below will be necessary to later proofs. 
\begin{defn} 
	Suppose $P$ and $Q$ are partitions. To \textit{add} the partitions $P+Q=R$,  combine the parts
	of $P$ and $Q$ into one partition R, reordering the parts as necessary to satisfy Definition
	\ref{Def}.\\
\end{defn}

This operation on partitions has been previously defined by Schneider as the \textit{product} of two partitions \cite{RS}.
Note that addition of partitions
is both associative and commutative. For example, given the partitions $A=[4,3,2,1]$ and $B=[3,2,2]$, we have $A+B=[4,3,3,2,2,2,1]=B+A$.\\

\begin{defn} 
	A \textit{sub-partition} $Q$ of a partition $P$ is a set of sequential parts of $P$ which is itself a partition and satisfies Definition \ref{Def}.
\end{defn}
This sub-partition $Q$, when correctly added to another sub-partition, or series of sub-partitions, will yield $P$ itself.
To sub-partition $P$ is to partition $P$ into sub-partitions. Thus, ``sub-partition" is both a noun and a verb. Sub-partition is the inverse operation of addition.

Furthermore, the empty partition $E=[\ ]$ will serve as an identity, so $E+P=P+E=P$ for any partition $P$.\\

For example, given $P=[a,b,c,d]$, we could sub-partition $P$ in some of the ways below (and yield the following valid sub-partitions):
\begin{enumerate}
	\item[1.] $P=[a,b,c,d]+[\ ]=[\ ]+[a,b,c,d]$
	\item[2.] $P=[a,b,c]+[d]$
	\item[3.] $P=[a,b]+[c,d]$
	\item[4.] $P=[a]+[b,c,d]$
	\item[5.] $P=[a,b]+[c]+[d]$
	\item[6.] $P=[a]+[b,c]+[d]$
	\item[7.] $P=[a]+[b]+[c,d]$
	\item[8.] $P=[a]+[b]+[c]+[d]$
	\item[9.] $P=[a,d]+[b,c]$
	\item[10.] $P=[a,c]+[b,d]$\\
\end{enumerate}

\section{Nesting}\label{Sec5}
	\begin{figure}[h]
	\centering
	\includegraphics[width=50mm]{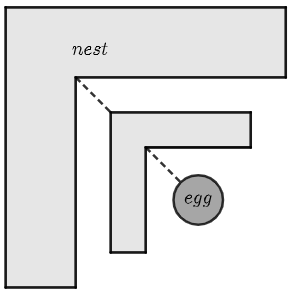}
	\caption{Nested Diagram of some partition $P$.}
	\label{figure8}
	\end{figure}

When looking at the Young diagram of some symmetric partition $P$, we can see a ``nest" of fancy L(s) and an ``egg" that is either a pure Durfee square or a fancy triangle (see Figure \ref{figure8} above for an illustration).
\begin{ex}
	Consider the self-conjugate partition $P=[7,7,6,4,3,3,1]$. Its Young diagram is given below.\\
	\begin{figure}[h]
		\centering
		\includegraphics[width=50mm]{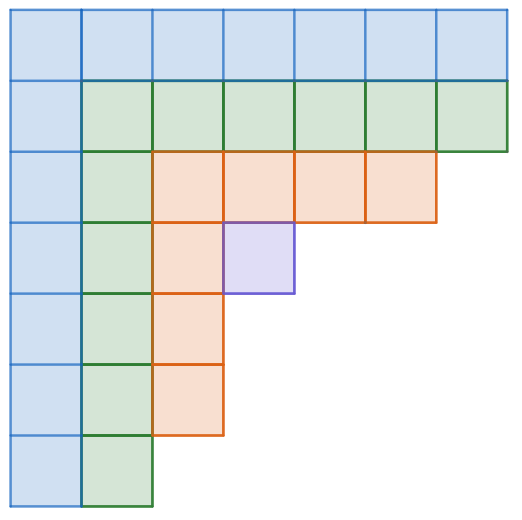}
	\end{figure}

	The blue, green, and orange areas are fancy Ls, and the purple area is a (trivial) Durfee square. The fancy Ls combined are the nest and the Durfee square is the egg. Now, we can combine the blue and green area to be a fancy L of width 2,
	as opposed to the fancy Ls of width 1 in Lemma \ref{Triangle}. These higher width fancy Ls are still
	symmetric.
\end{ex}

We want to show that if both the nest and the egg are truly symmetric, then we will have a self-conjugate partition.  We first prove the trivial cases, where we only have a single fancy L of width 1 nesting a symmetric egg (first, a pure Durfee square, and second, a fancy triangle).
\subsection{The Trivial Cases}

\begin{lem} 
	Suppose the first part of a partition $P$ is equal to the number of parts of $P$, and 
	\[P=[n,\ m,\ m,...,\ m,\ 1,\ 1, ..., 1].\]
	\begin{figure}[h]
		\centering
		\includegraphics[width=35mm]{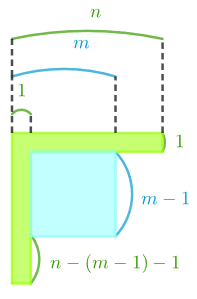}
	\end{figure}
	
	Sub-partition $P$ as such:
	\[P=:[n]+[m,\ m,...,\ m]+[1,\ 1,..., 1]=X+Y+Z,\]
	respectively. Suppose $Y$ has $m-1$ parts and $Z$ has $n-(m-1)-1$ parts. Then, $P$ is self-conjugate.\\
\end{lem}

\begin{proof}
	Let $P$ be as above. Remove the first (outer) row and column. Then we are left with the partition
	\[Q=:[(n-n)]+[(m)-1,\ (m)-1,...,\ (m)-1]+[(1-1),\ (1-1),...,\ (1-1)]\]
	\[=[\ ]+[(m-1),\ (m-1),...,\ (m-1)]+[\ ]\ \ \ \ \ \ \ \ \ \ \ \ \ \ \ \ \ \ \ \ \ \ \ \ \ \ \ \ \ \ \ \ \ \ \ \ \]
	\[=E+[(m-1),\ (m-1),...,\ (m-1)]+E\ \ \ \ \ \ \ \ \ \ \ \ \ \ \ \ \ \ \ \ \ \ \ \ \ \ \ \ \ \ \ \ \ \ \ \ \]
	\[=[m-1,\ m-1,...,\ m-1].\ \ \ \ \ \ \ \ \ \ \ \ \ \ \ \ \ \ \ \ \ \ \ \ \ \ \ \ \ \ \ \ \ \ \ \ \ \ \ \ \ \ \ \ \ \ \ \ \ \ \ \ \ \ \]
	So, $Q$ is a pure Durfee square, of dimension $m-1$, and is self-conjugate by Lemma \ref{Durfee}. The first (outer) row and column was a fancy L, which is self-conjugate by Lemma \ref{L}. Thus, the Young diagram of $P$ is symmetric, so $P$ is self-conjugate.\\
\end{proof}

\begin{lem} 
	Suppose the first part of a partition $P$ is equal to the number of parts of $P$ and
	\[P=[n,\ m,\ (m-1),...,\ (m-(m-2)),\ \ 1,\ 1, ..., 1].\]
	\begin{figure}[h]
		\centering
		\includegraphics[width=60mm]{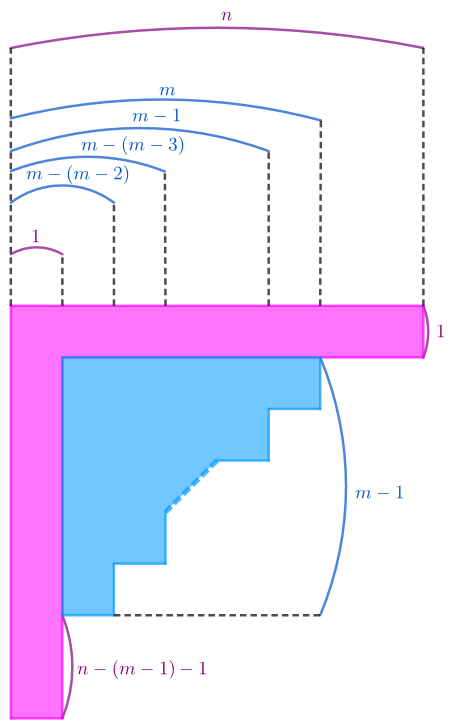}
	\end{figure}
	
	Sub-partition $P$ as such:
	\[P=:[n]+[m,\ (m-1),...,\ (m-(m-2))]+[1,\ 1, ..., 1]=X+Y+Z,\]
	respectively. Suppose $Y$ has $m-1$ parts and $Z$ has $n-(m-1)-1$ parts. Then $P$ is self-conjugate.\\
\end{lem}

\begin{proof}
	Let $P$ be as above. Remove the first (outer) row and column. Then we are left with the partition
	\[Q=:[(n-n)]+[m-1,\ (m-1)-1,...,\ (m-(m-2))-1]+[(1-1),...,\ (1-1)]\]
	\[=[\ ]+[m-1,\ (m-1)-1,...,\ (m-(m-2))-1]+[\ ]\ \ \ \ \ \ \ \ \ \ \ \ \ \ \ \ \ \]
	\[=E+[m-1,\ (m-1)-1,...,\ (m-(m-2))-1]+E\ \ \ \ \ \ \ \ \ \ \ \ \ \ \ \ \ \]
	\[=[m-1,\ (m-1)-1,...,\ (m-1)-((m-1)-1)].\ \ \ \ \ \ \ \ \ \ \ \ \ \ \ \ \ \ \ \ \ \ \ \]
	So, $Q$ is a fancy triangle and is self-conjugate by Lemma \ref{Triangle}. The first (outer) row and column was a fancy L, which is self-conjugate by Lemma \ref{L}. Thus, the Young diagram of $P$ is symmetric, so $P$ is self-conjugate.\\
\end{proof}

We now want to show that if both the nest and the egg are symmetric, then we will have a self-conjugate partition in general, where our nests are non-trivial multiple fancy Ls. We will show this is true for eggs that are pure Durfee squares and fancy triangles.
\subsection{The Non-Trivial Cases}
\begin{lem} 
	\label{NonTrivSquare}
	Suppose the first part of a partition $P$ is equal to the number of parts of $P$ and 
	\[P=[n,\ n,...,\ n,\ m,\ m,...,\ m,\ j,\ j,...,\ j].\]
	\begin{figure}[h]
		\centering
		\includegraphics[width=37mm]{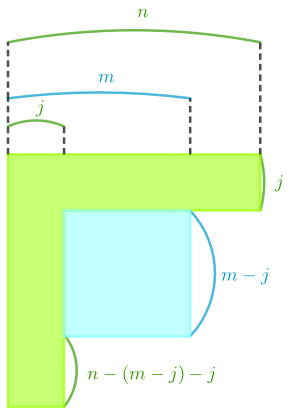}
	\end{figure}
	
	Sub-partition $P$ as such:
	\[P=:[n,\ n,...,\ n]+[m,\ m,...,\ m]+[j,\ j,...,\ j]=X+Y+Z,\]
	respectively. Suppose $X$ has $j$ parts, $Y$ has $m-j$ parts, and $Z$ has $n-(m-j)-j$ parts. Then $P$ is self-conjugate.\\
\end{lem}

\begin{proof}
	Let $P$ be as above. Remove the first (outer) $j$ rows and $j$ columns. You will be left with:
	\begin{enumerate}
		\item[$Q=:$] $[(n-n),\ (n-n),...,\ (n-n)]+[m-j,\ m-j,...,\ m-j]\\ +[(j-j),\ (j-j),...,\ (j-j)]$
		\item[]$=[\ ]+[m-j,\ m-j,...,\ m-j]+[\ ]$
		\item[]$=E+[m-j,\ m-j,...,\ m-j]+E$
		\item[]$=[m-j,\ m-j,...,\ m-j]$
	\end{enumerate}
	So, $Q$ is a Durfee square and is self-conjugate by Lemma \ref{Durfee}. The first  (outer) $j$ rows and $j$ columns were a fancy L of width $j$, which is self-conjugate by iterating Lemma \ref{L} $j$ times. Thus, the Young diagram of $P$ is symmetric, so $P$ is self-conjugate.\\
\end{proof}

\begin{lem} 
	\label{NonTrivTriangle}
	Suppose the first part of a partition $P$ is equal to the number of parts of $P$ and 
	\[P=[n,\ n,...,\ n,\ m,\ m-1,...,\ m-(m-j)+1,\ j,\ j,...,\ j].\]
	\begin{figure}[h]
		\centering
		\includegraphics[width=60mm]{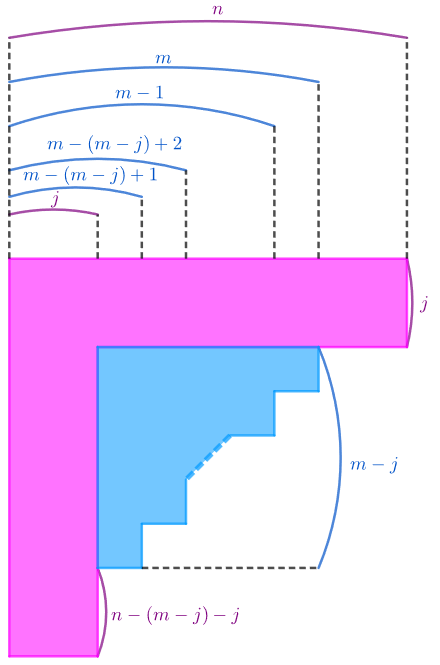}
	\end{figure}
	
	Sub-partition $P$ as such:
	\[P=:[n,\ n,...,\ n]+[m,\ m-1,...,\ m-(m-j)+1]+[j,\ j,...,\ j]=X+Y+Z,\]
	respectively. Suppose $X$ has $j$ parts, $Y$ has $m-j$ parts, and $Z$ has $n-(m-j)-j$ parts. Then $P$ is self-conjugate.\\
\end{lem}

\begin{proof}
	Let $P$ be as above. Remove the first (outer) $j$ rows and $j$ columns. Then we are left with the partition
	\begin{enumerate}
		\item[$Q$] $=[(n-n),\ (n-n),...,\ (n-n)] $
		\item[]$\ \ \ +[m-j,\ (m-1)-j,...,\ (m-(m-j)+1)-j]$
		\item[]$\ \ \ +[(j-j),\ (j-j),...,\ (j-j)]$
		\item[]$=[\ ]+[(m-j),\ (m-j)-1,...,\ ((m-j)-(m-j))+1]+[\ ]$
		\item[]$=E+[(m-j),\ (m-j)-1,...,\ ((m-j)-(m-j))+1]+E$
		\item[]$=[(m-j),\ (m-j)-1,...,\ (m-j)-((m-j)-1)]$
	\end{enumerate}
	So, $Q$ is a fancy triangle and is self-conjugate by Lemma \ref{Triangle}, and the (outer) $j$ rows and $j$ columns were a fancy L of width $j$, which is self-conjugate by iterating Lemma \ref{L} $j$ times. Thus, the Young diagram of $P$ is symmetric, so $P$ is self-conjugate.\\
\end{proof}

\subsection{The Nest-And-Egg Theorem}
\begin{thm} 
	\label{Nest-And-Egg}
	Given a nest, which is a series of fancy Ls of any width, and an egg that is either a Durfee square or a fancy triangle of any size $($or an empty egg, $E=[\ ])$, a partition composed of this nest and egg will always be self-conjugate.\\
\end{thm}

\begin{proof}
	Suppose we have a nest and a trivial egg; the egg $K$ is either $K=[\ ]=E$, $K=[1]$, $K=[2,1]$, or $K=[2,2]$, and the rest of our Young diagram belongs to the nest (and thus, is some fancy L or a series of fancy Ls of any width). Since all Ls of the nest are symmetric by Lemma \ref{L} and all of these egg cases are trivially symmetric, the claim holds true.
	Suppose we have a nest and a non-trivial egg, such as a Durfee square or fancy triangle of dimension 3 or more. By Lemma \ref{NonTrivSquare} and Lemma \ref{NonTrivTriangle}, the claim holds true, since removing enough outer fancy Ls from our Young diagram will leave us with one of these cases. Thus, the theorem holds true in all cases.\\
\end{proof}

\section{The Main Theorem} \label{Sec6}

With all of the aforementioned lemmas and theorems, we finally wish to prove the main theorem: the Self-Conjugate Partition Theorem. This theorem will allow us to take any partition, without its Young diagram, and determine if it is self-conjugate, even if it does not follow any of the forms or patterns above. Although the statement of the theorem and its proof may appear complicated, it will show us a simple way to determine whether a partition is self-conjugate by only using clever counting skills and summation.

However, we note that the $x_i$-values in the next theorem represent the \textit{multiplicity} of parts (how many times a part repeats in the partition) and not the value or index of the parts themselves. Figure \ref{figure9} will be a useful illustration for the proof of the theorem.

\begin{figure}[h]
	\centering
	\includegraphics[width=65mm]{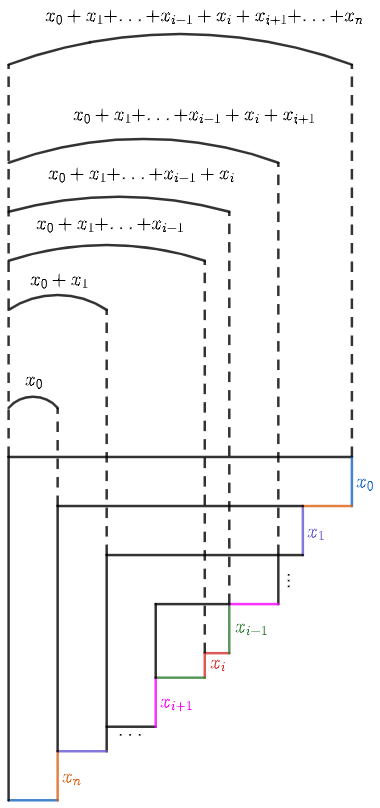}
	\caption{The Young diagram of a self-conjugate partition $P$.}
	\label{figure9}
\end{figure}

\begin{thm} [The Self-Conjugate Partition Theorem]
	A partition $P$ is self-conjugate if and only if it can be written as the following:
	\[P=\left[\sum_{i=0}^{n}x_i,\ \sum_{i=0}^{n}x_i,...,\sum_{i=0}^{n}x_i,\ \sum_{i=0}^{n-1}x_i,\ \sum_{i=0}^{n-1}x_i,...,\sum_{i=0}^{n-1}x_i,...,\sum_{i=0}^{0}x_i,\ \sum_{i=0}^{0}x_i,...,\sum_{i=0}^{0}x_i,\right]\]
	\[=\left[\sum_{i=0}^{n}x_i,\ \sum_{i=0}^{n}x_i,...,\sum_{i=0}^{n}x_i\right]+\left[\sum_{i=0}^{n-1}x_i,\ \sum_{i=0}^{n-1}x_i,...,\sum_{i=0}^{n-1}x_i\right]+...+\left[\sum_{i=0}^{0}x_i,\ \sum_{i=0}^{0}x_i,...,\sum_{i=0}^{0}x_i\right]\]
	\[=: Y_n+Y_{n-1}+\cdots+Y_0,\]
	respectively, where $P$ has $\sum_{i=0}^{n}x_i$ parts, $Y_n$ has $x_0$ parts, $Y_{n-1}$ has $x_1$ parts, ..., $Y_0$ has $x_n$ parts, and all $x_i\in\N$.
\end{thm}

\begin{proof}$\ $
	\begin{enumerate}
		\item[ ]
		Suppose a partition $P$ is self-conjugate. We write
		\[P=[z_n,\ z_n,\ ...,\ z_n,\ z_{n-1},\ z_{n-1},\ ...,\ z_{n-1},\ ...,\ z_0,\ z_0,\ ...,\ z_0]\]
		\[=[z_n,\ z_n,\ ...,\ z_n]+[z_{n-1},\ z_{n-1},\ ...,\ z_{n-1}]+...+[z_0,\ z_0,\ ...,\ z_0]\]
		\[=:Y_n+Y_{n-1}+\cdots+Y_0,\]
		where $P$ has $z_n$ parts. It must be that $Y_n$ has $z_0$ parts, lest we contradict Theorem \ref{Piecewise}, since $P$ is self-conjugate, and the boundaries of our piecewise function have to be equal on either side of $g(x)=x$ (see Figure \ref{figure9}). Likewise, $Y_{n-1}$ has $z_1-z_0$ parts, lest we contradict Theorem \ref{Piecewise}. Continuing inductively, we find that
		\[\ell(Y_n)=z_0,\ \ell(Y_{n-1})=z_1-z_0,\ \ell(Y_{n-2})=z_2-z_1,\ ...,\ \ell(Y_0)=z_n-z_{n-1}.\ \]
		
		Let $x_i$ denote the number of parts of size $z_{n-1}$ in $P$, for all $0\leq i \leq n$. Again, by Theorem \ref{Piecewise}, since the boundaries must match, we have that
		\[z_n=\sum_{i=0}^{n}x_i,\ z_{n-1}=\sum_{i=0}^{n-1}x_i,\ ..., \ z_0=\sum_{i=0}^{0}x_i. \]
		
		Thus, if $P$ is a self-conjugate partition, it can be written as the above partition.
		
		\item[ ]
		Conversely, suppose we have the partition
		\[P=\left[\sum_{i=0}^{n}x_i,\ \sum_{i=0}^{n}x_i,...,\sum_{i=0}^{n}x_i,\ \sum_{i=0}^{n-1}x_i,\ \sum_{i=0}^{n-1}x_i,...,\sum_{i=0}^{n-1}x_i,...,\sum_{i=0}^{0}x_i,\ \sum_{i=0}^{0}x_i,...,\sum_{i=0}^{0}x_i,\right].\] 
		Remove the first $x_0$ (outer) rows and columns. We are left with:
		
		\[Q=\left[\left(\sum_{i=0}^{n}x_i\right)-\left(\sum_{i=0}^{n}x_i\right),\left(\sum_{i=0}^{n}x_i\right)-\left(\sum_{i=0}^{n}x_i\right),...,\left(\sum_{i=0}^{n}x_i\right)-\left(\sum_{i=0}^{n}x_i\right)\right]\]
		
		\[+\left[\left(\sum_{i=0}^{n-1}x_i\right)-x_0,\ \left(\sum_{i=0}^{n-1}x_i\right)-x_0,...,\left(\sum_{i=0}^{n-1}x_i\right)-x_0\right]+\cdots\]
		
		
		\[=[\ ]+\left[\sum_{i=1}^{n-1}x_i,\sum_{i=1}^{n-1}x_i,...,\sum_{i=1}^{n-1}x_i\right]+\cdots+\left[\sum_{i=1}^{1}x_i,\sum_{i=1}^{1}x_i,...,\sum_{i=1}^{1}x_i\right]+[\ ]\]
		
		\[=\left[\sum_{i=1}^{n-1}x_i,\sum_{i=1}^{n-1}x_i,...,\sum_{i=1}^{n-1}x_i\right]+\left[\sum_{i=1}^{n-2}x_i,\sum_{i=1}^{n-2}x_i,...,\sum_{i=1}^{n-2}x_i\right]+\cdots+\left[\sum_{i=1}^{1}x_i,\sum_{i=1}^{1}x_i,...,\sum_{i=1}^{1}x_i\right].\]
		Let $j=i-1$ and $m=n-1$. Then,
		\[Q=\left[\sum_{j=0}^{m}x_j,\sum_{j=0}^{m}x_j,...,\sum_{j=0}^{m}x_j\right]+\left[\sum_{j=0}^{m-1}x_j,\sum_{j=0}^{m-1}x_j,...,\sum_{j=0}^{m-1}x_j\right]+\cdots+\left[\sum_{j=0}^{0}x_j,\sum_{j=0}^{0}x_j,...,\sum_{j=0}^{0}x_j\right].\]
		
		We are now back where we started. Continue this process inductively, removing every next outer set of $x_k$ rows and columns, for $k\in\{1,2,...,n\}$. This will peel back our nest, and by Theorem \ref{Nest-And-Egg}, we will eventually get a symmetric egg. This is because every layer we peel back is a fancy L, and the egg itself is either empty or made of fancy Ls (Durfee square, fancy triangle). Alternatively, assume the supposition above and one or more parts of the partition makes the partition a non-self-conjugate; this will lead to a contradiction because of our supposition on the multiplicity of parts. Thus, $P$ is a self-conjugate partition.
	\end{enumerate}
\end{proof}

Now, let us go through a few examples and apply this theorem to see if various partitions are self-conjugate.

\begin{ex}
	Let $P=[5,5,3,2,2]$.
	We see that $x_0=2$, $x_1=1$, $x_2=2$, since 5 appears twice, 3 once, and 2 twice.
	We have
	\[5=\sum_{i=0}^{2}x_i=2+1+2,\ \ 3=\sum_{i=0}^{1}x_i=2+1,\ \ 2=\sum_{i=0}^{0}x_i=2,\]
	so $P$ is self-conjugate.\\
\end{ex}

\begin{ex}
	Let $Q=[20,7,7,7,7,7,6,6,5,5,5,5,4,4,2,2,1,1,1,1]$. We see that $x_0=1$, $x_1=5$, $x_2=2$, $x_3=4$, $x_4=2$, $x_5=2$, and $x_6=4$. We have
	\[\sum_{i=0}^{6}x_i=1+5+2+4+2+2+4=20,\ \ \sum_{i=0}^{5}x_i=1+5+2+4+2+2=16\neq7,\]
	so $Q$ is not self-conjugate.\\
\end{ex}

\begin{ex}
	Let $R=[9,9,7,7,4,4,4,2,2]$. We see that $x_0=2$, $x_1=2$, $x_2=3$, $x_3=2$. We have
	\[\sum_{i=0}^{3}x_i=2+2+3+2=9,\ \ \sum_{i=0}^{2}x_i=2+2+3=7,\ \ \sum_{i=0}^{1}x_i=2+2=4, \ \ \sum_{i=0}^{0}x_i=2,\]
	so $R$ is self-conjugate.\\\\
\end{ex}

\end{document}